\documentclass[preprint]{elsarticle}
\usepackage{lineno}

\usepackage{stmaryrd}
\usepackage{slashbox}
\usepackage{amsmath}
\usepackage{amssymb}
\usepackage{amsthm}
\usepackage{graphicx}
\usepackage{epic,eepic,epsfig}
\usepackage{color}
\usepackage{subfigure}  
\usepackage{placeins}   
\usepackage{multirow}
\usepackage{epstopdf}
\usepackage{varwidth}
\usepackage{float}
\usepackage[table]{xcolor}

\newtheorem{theorem}{Theorem}[section]
\newtheorem{example}{Example}[section]
\newtheorem{lemma}{Lemma}[section]

\newtheorem{definition}{Definition}
\newcommand{\Lbrace}{ \left\{\kern -0.23em\left\{     }
 \newcommand{\Rbrace}{ \right\}\kern -0.23em\right\} }
\newcommand{\Lbracket }{ \left[\kern -0.09em\left[    }
 \newcommand{\Rbracket }{ \right] \kern-0.09em\right] }

\definecolor{tabclr}{cmyk}{0,0,1,0}

\begin{document}

\title{A modified weak Galerkin method for $\boldsymbol{H}(\mathrm{curl})$-elliptic problem}

\author[SCNU]{Ming Tang} 
\ead{mingtang@m.scnu.edu.cn}

\author[SCNU]{Liuqiang Zhong}
\ead{zhong@scnu.edu.cn}

\author[GU]{Yingying Xie\corref{cor}}
\ead{xieyy@gzhu.edu.cn}

\cortext[cor]{Corresponding author}
\address[SCNU]{School of Mathematical Sciences, South China Normal University, Guangzhou 510631, China}
\address[GU]{School of Mathematics and Information Science, Guangzhou University, Guangzhou 510006,China}

\begin{abstract}
In this paper, we design and analysis a  modified weak Galerkin (MWG) finite element method for $\boldsymbol{H}(\mathrm{curl})-$elliptic problem. We first  introduce a new discrete weak curl operator and the MWG finite element space.  The modified weak Galerkin method does not require the penalty parameter by comparing with traditional DG methods. We prove optimal error estimates in energy norm. At last, we provide the numerical results to confirm these theoretical results.
\end{abstract}


\maketitle

\section{Introduction}

Let $\Omega\subset \mathbb{R}^2$ be a bounded Lipschitz polyhedron  with boundary $\partial\Omega$. We consider the following $\boldsymbol{H}(\mathrm{curl})$ elliptic problem:
\begin{eqnarray}\label{Equ:1.1}
	\mathbf{curl}\mathrm{curl} \boldsymbol{u} + \boldsymbol{u}=\boldsymbol{f}, \quad \mbox{in}\ \Omega,\\ \label{Equ:1.2}
	\boldsymbol{u}\cdot\boldsymbol{t} =0, \quad \mbox{on}\ \partial \Omega,
\end{eqnarray}
where $\boldsymbol{t}$ is the unit tangent on $\partial\Omega$ oriented counter-clockwisely, $\boldsymbol{u}$ is the electric or magnetic field, $\boldsymbol{f}$ is a given vector field depending on a given external source field.
We recall that, $\mathrm{curl}\  \boldsymbol{v}={\partial v_2}/{\partial x}-{\partial v_1}/{\partial y}$ for a vector field $\boldsymbol{v}=(v_1, v_2)$, while $\mathbf{curl}\ \phi =({\partial \phi}/{\partial y}, -{\partial \phi}/{\partial x})$ for a scalar function $\phi$.

Weak Galerkin (WG) method was introduced by Wang and Ye \cite{WangJPYeX13:103} for solving a second order elliptic problem. The weak function $v=\{v^0,v^b\}$ was used as the approximate function space, and weak gradient operator was used instead of classical gradient operator. The first component $v^0$ can be understood as the value of $v$ in the interior of elements, and the second component $v^b$ can be understood as the value of $v$ on the boundary of elements.
WG method can introduce stabilization terms in bilinear form to describe the continuity of the weak functions.
Moreover,
WG method can also be applied in arbitrary polygon or polyhedron partitions and has strong flexibility and robustness, whose corresponding bilinear form is independent on parameters \cite{MuLWangJP15:45}. Hence, WG method has been developed to solve many equations, such as the Stokes equations \cite{WangJPYeX16:155, WangRSWangXS16:171, ZhaiZhang15:2455}, the Helmholtz equation \cite{MuWang14:1461, MuLWangJP15:1228, DuYZhangZM17:133}, the biharmonic equation \cite{MuLWangJP14:1003, WangCMWangJP13:2314, MuLWangJP13:247, MuLWangJP14:473, WangCMWangJP15:302}, parabolic equations \cite{LiQLWangJP13:2004, ZhangHQZouYK16:24},  Navier-Stokes equations \cite{ZhangJCZhangK18:706}, Oseen equations \cite{LiuXLiJ16:1473},  Brinkman equation \cite{MuLWangJP14:327}, Darcy-Stokes equation \cite{ChenWBWangF16:897} and Darcy equation \cite{LinGLiuJG14:422}. Recently,
there are also some works with Maxwell equations \cite{MuLWangJP15:363, ShieldsLi17:2106, WangCM18:127}.
For example, Mu, Wang, Ye and Zhang \cite{MuLWangJP15:363} applied the WG method to the time-harmonic Maxwell equations, and gave a numerical method with the optimal order of convergence in certain discrete norms. Shields, Li and Machorro  \cite{ShieldsLi17:2106} developed the WG method to solve the time-dependent Maxwell equations, and proved the optimal order of convergence in the energy norm. Wang \cite{WangCM18:127} introduced a new discrete scheme of time harmonic Maxwell equations in connected domain based on WG method, and established the error estimates of the optimal order of various discrete Sobolev norms.

Modified weak Galerkin(MWG) method was put forward by Wang, Malluwawadu, Gao and McMillan \cite{WangMalluwawadu14:319} for elliptic problem.
The motivation of MWG method is to reduce the number of unknowns, in which $v^b$ is replaced by the average of $v^0$.
Comparing with WG method, MWG method contains less unknowns, while the accuracy stays the same. 
Then, MWG method has also found its way to other problems, such as the parabolic problems \cite{GaoFZWangXS14:1},
Sobolev equation \cite{GaoFZWangXS15:307},
Signorini and obstacle problems \cite{ZengYPChenJR17:1459}, Stokes problem \cite{TianTZhaiQL18:268, MuLWangXS15:79}
and poroelasticity problem \cite{WangRSWangXS18:518}.
However, to our best knowledge, there are not any  published  literatures for  the MWG discretization  of $\boldsymbol{H}(\mathrm{curl})$ elliptic problem.
This manuscript  will provide a MWG method for the $\boldsymbol{H}(\mathrm{curl})$ elliptic problem \eqref{Equ:1.1}-\eqref{Equ:1.2}, prove both the corresponding well-posedness and the optimal order of convergence in the energy norm. The optimal order of convergence is verified by numerical experiments.

In this manuscript, we shall follow the state-of-the-art theoretical analysis in \cite{ZhangLin17:381} to present the error estimation for the MWG method of $\boldsymbol{H}(\mathrm{curl})$ elliptic problem.
We stress
that the extension of the theory to MWG method of $\boldsymbol{H}(\mathrm{curl})$ elliptic problem is not straightforward,  this paper's contributions include: (1) modified weak curl operator instead of modified grad operator is used, here we explore the relationship between the modified weak curl operator and classical curl operator.
(2) lower order term of  the $\boldsymbol{H}(\mathrm{curl})$ elliptic problem,  which is missed in  the second order elliptic problem, should be  considered carefully. To conquer these problem, we establish  some technical equations.


The rest of the article is organized as follows.
In Sect. \ref{sec:2}, we describe the definitions of
weak curl and modified weak curl, the corresponding modified weak Galerkin algorithm for model problem \eqref{Equ:1.1}-\eqref{Equ:1.2} and the corresponding modified weak Galerkin (MWG) algorithm.
In Sect. \ref{sec:3}, we estimate the error in the energy norm for above MWG methods.
In Sect. \ref{sec:5},  we provide numerical experiments to support the theoretical analysis.


\section{The modified weak Galerkin method}\label{sec:2}
In this section, we will define a modified weak curl, introduce a modified weak Galerkin methods for \eqref{Equ:1.1}-\eqref{Equ:1.2} and present the corresponding well-posedness.

First, we present some notations.  Let $K$ be any polygonal domain in $\mathbb{R}^2$ with boundary $\partial K$, we denote by $L^2(K)$ the Hilbert space of square integrable functions
fields with inner product $(\cdot, \cdot)_K$ and the corresponding norm $\|\cdot\|_{0;K}^2 := (\cdot, \cdot)_{K}$. And we use the standard definitions Sobolev spaces $H^s(K)$ and the corresponding norms $\|\cdot\|_{s, K}$ with real number $s\geqslant 0$ and $H_0^1(K) = \{v:v\in H^1(K), v|_{\partial K} = 0\}$, where $v|_{\partial K}$ is in the sense of trace.
In particular, when $K=\Omega$, we simplify $\|\cdot\|_{s, K}$ as $\|\cdot\|_{s}$.
We define
$\boldsymbol{L}^2(K) = [L^2(K)]^2$, $\boldsymbol{H}^s(K)=[H^s(K)]^2$,  $\boldsymbol{H}_0^1(K)=[H_0^1(K)]^2$ and
\begin{eqnarray*}
	&\boldsymbol{H}(\mathrm{curl}; K) =\{\boldsymbol{v}: \boldsymbol{v}\in \boldsymbol{L}^2(K),  \mathrm{curl}\  \boldsymbol{v}\in L^2(K)\},
\end{eqnarray*}
Additionally, we define the subspace of $\boldsymbol{H}(\mathrm{curl}; K)$ as follows:
\begin{equation*}
	\boldsymbol{H}_0(\mathrm{curl}; K) =\{\boldsymbol{v}\in \boldsymbol{H}(\mathrm{curl}; K), \boldsymbol{v}\cdot\boldsymbol{t}=0\ \mbox{on} \ \partial K\}.
\end{equation*}


Given a shape-regular triangulation $\mathcal{T}_h$ for $\Omega$, For $\tau\in\mathcal{T}_h$, we write $h_{\tau} = |\tau|^{1/2} $ to denote
the local mesh size of the element $\tau$ where $|\tau|$ is the Lebesgue measure of $\tau$ and $h=\max_{\tau\in\mathcal{T}_h} h_{\tau}$.

We denotes $\mathcal{E}_h$ the set of all the edges, $\mathcal{E}_h^0$ denotes the set of all the interior edges, $\mathcal{E}_h^{\partial}$ denotes the set of all the boundary edges.
For $\mathcal{T}^{\prime}_h\subseteq \mathcal{T}_h$ and $\mathcal{E}_h^{\prime}\subseteq \mathcal{E}_h$, we introduce the discrete $L^2$ inner product and norm:
\begin{eqnarray*}
	& \displaystyle
	(\boldsymbol{v}, \boldsymbol{w})_{\mathcal{T}^{\prime}_h}
	=\sum\limits_{\tau \in \mathcal{T}^{\prime}_h}(\boldsymbol{v}, \boldsymbol{w})_{\tau}
	=\sum\limits_{\tau\in \mathcal{T}^{\prime}_h} \int_{\tau} \boldsymbol{v}\cdot \boldsymbol{w} \mathrm{d} x, \quad\|\boldsymbol{v}\|_{\mathcal{T}^{\prime}_h}^{2}
	=(\boldsymbol{v}, \boldsymbol{v})_{\mathcal{T}^{\prime}_h},
	\\
	& \displaystyle
	\langle \boldsymbol{v}, \boldsymbol{w} \rangle_{\mathcal{E}_h^{\prime}}
	=\sum\limits_{e \in \mathcal{E}_h^{\prime}} \langle \boldsymbol{v}, \boldsymbol{w} \rangle_{e}
	=\sum\limits_{e \in \mathcal{E}_h^{\prime}} \int_{e} \boldsymbol{v}\cdot \boldsymbol{w}\mathrm{d} s, \quad \|\boldsymbol{v}_h\|_{\mathcal{E}_h^{\prime}}^{2}
	=\langle \boldsymbol{v}, \boldsymbol{v} \rangle_{\mathcal{E}_h^{\prime}},
\end{eqnarray*}
where $\langle \boldsymbol{v}\cdot \boldsymbol{w} \rangle_{e}= \int_{e} \boldsymbol{v}\cdot \boldsymbol{w}\mathrm{d} s$ denotes the $L^2$ inner product on the edge $e$.

For any $e\in\mathcal{E}_h^{0}$, there exist two adjacent elements $\tau_1, \tau_2\in \mathcal{T}_h$ sharing the common edge $e = \partial\tau_1\cap\partial\tau_2$, we denote by $\boldsymbol{t}_1$ and $\boldsymbol{t}_2$ the unit tangential vectors on $e$ for $\tau_1$ and $\tau_2$, respectively. The average and tangential jump for a vector function $\boldsymbol{w}$ are defined as
\begin{eqnarray}\label{vectorinterior}
	 \Lbrace\boldsymbol{w}\Rbrace_e = (\boldsymbol{w}|_{\tau_1} + \boldsymbol{w}|_{\tau_2})/{2}, \
	\Lbracket \boldsymbol{w}\Rbracket_e= \boldsymbol{w}|_{\tau_1}\cdot\boldsymbol{t}_1 + \boldsymbol{w}|_{\tau_2}\cdot\boldsymbol{t}_2;
\end{eqnarray}

For a scalar function $\phi$, its average and tangential jump on $e$ are defined as
\begin{eqnarray}\label{average}
\Lbrace \phi \Rbrace_e = (\phi|_{\tau_1} + \phi|_{\tau_2})/{2},\ \Lbracket \phi\Rbracket_e =\phi|_{\tau_1} \boldsymbol{t}_1 + \phi|_{\tau_2} \boldsymbol{t}_2,
\end{eqnarray}
where $\boldsymbol{w}|_{\tau_i}$ and $\phi|_{\tau_i}$  denote the value of $\boldsymbol{w}$ on $\tau_i$, $i=1,2$.

For any $e\in\mathcal{E}_h^{\partial}$, there is a element $\tau\in\mathcal{T}_h$ such that $e\in\partial\tau \cap\partial\Omega$, we define the average and tangential jump for a vector function $\boldsymbol{w}$ are defined as
\begin{eqnarray}\label{vectorboundary}
	 \Lbrace \boldsymbol{w}\Rbrace_e=\boldsymbol{w}|_{\tau}, \  \Lbracket \boldsymbol{w}\Rbracket_e=\boldsymbol{w}|_\tau\cdot\boldsymbol{t}.
\end{eqnarray}
For a scalar function $\phi$, its average and  tangential jump on $e$ are defined as
\begin{eqnarray} \label{jump}
 \Lbrace \phi \Rbrace_e=\phi|_{\tau}, \
 \Lbracket \phi\Rbracket_e=\phi|_{\tau}\boldsymbol{t},
\end{eqnarray}
where $\boldsymbol{w}|_{\tau}$ and $\phi|_{\tau}$ denote the value of $\boldsymbol{w}$ on $\tau$.

\subsection{The modified discrete weak curl}
In order to describe the modified weak curl, we recall the concepts of weak curl and discrete weak curl.
A weak function on the element $\tau$ refers to a function $\boldsymbol{v}=\{\boldsymbol{v}^0, \boldsymbol{v}^b\}$ where $\boldsymbol{v}^0 \in \boldsymbol{L}^2(\tau),\boldsymbol{v}^b \cdot\boldsymbol{t} \in L^2(\partial \tau)$.
The first component $\boldsymbol{v}^0$ can be understood as the value of $\boldsymbol{v}$ in the interior of $\tau$, and the second component $\boldsymbol{v}^b$ is the value of $\boldsymbol{v}$ on the boundary of $\tau$. We denote the space of weak functions on $\tau$ as:
\begin{equation*}\label{equ:weakfunction1}
	V(\tau):=\{\boldsymbol{v}=\{\boldsymbol{v}^0, \boldsymbol{v}^b\}: \boldsymbol{v}^0\in  \boldsymbol{L}^2(\tau), \boldsymbol{v}^b\cdot\boldsymbol{t}\in L^2(\partial  \tau)\},
\end{equation*}
where $\boldsymbol{t}$ is the unit tangent on $\partial \tau$ oriented counter-clockwisely.

\begin{definition}[Weak curl]
	For any $\boldsymbol{v}\in V(\tau)$, the weak curl of $\boldsymbol{v}$ is defined as a continuous linear functional
	$\mathrm{curl}^c_{w}\ \boldsymbol{v}\in H^1(\tau)$ whose action on each $\phi\in H^1(\tau)$ is given by
	\begin{equation*}\label{equ:weakcurl}
		(\mathrm{curl}^c _{w}\ \boldsymbol{v}, \phi)_\tau:=(\boldsymbol{v}^0, \mathbf{curl}\ \phi)_\tau + \langle\boldsymbol{v}^b\cdot\boldsymbol{t}, \phi\rangle_{\partial \tau},
	\end{equation*}
	where $(\cdot, \cdot)_\tau$ is the $L^2$ inner product on $\tau$ and $\langle\cdot, \cdot\rangle_{\partial \tau}$ is the $L^2$ inner product on $\partial \tau$.
\end{definition}

For a positive constant $k\geq1$, we denote by $P_{k}(\tau)$ the set of polynomials on $\tau$ with degree no more than $k$.

\begin{definition}[Discrete Weak Curl]\label{def:dweakcurl}
	The discrete weak curl operator, denoted by $\mathrm{curl}^d _{w,k-1, \tau}$, is defined as the unique polynomial
	$\mathrm{curl}^d _{w,k-1, \tau}\ \boldsymbol{v}\in P_{k-1}(\tau)$ for any $\boldsymbol{v} \in V(\tau)$ that satisfies:
	\begin{equation*}\label{equ:dweakcurl}
		(\mathrm{curl}^d _{w,k-1, \tau}\ \boldsymbol{v}, \phi)_\tau:=
		(\boldsymbol{v}^0, \mathbf{curl}\ \phi)_\tau+\langle\boldsymbol{v}^b\cdot\boldsymbol{t}, \phi\rangle_{\partial \tau}, \forall \phi\in P_{k-1}(\tau).
	\end{equation*}
\end{definition}
The definitions of weak curl and discrete weak curl for 3D can be found in \cite{MuLWangJP15:363} or \cite{ShieldsLi17:2106}, respectively.

Similar to the definition of modified weak gradient in \cite{WangMalluwawadu14:319},  for the function $\boldsymbol{v}=\{\boldsymbol{v}^0, \boldsymbol{v}^b\}\in V(\tau)$, the selected component $\boldsymbol{v}^b=\Lbrace \boldsymbol{v}^0\Rbrace_e$ is not an independent component because $\boldsymbol{v}^b=\Lbrace \boldsymbol{v}^0\Rbrace_e$ is determined by $\boldsymbol{v}^0$. Futhermore, we define a discrete modified weak curl as follows.

\begin{definition}[Discrete Modified Weak Curl]\label{def:modifiedweakcurl}
	The discrete modified weak curl operator, denoted by $\mathrm{curl} _{w,k-1, \tau}$, is defined as the unique polynomial
	$\mathrm{curl} _{w,k-1, \tau}\ \boldsymbol{v}\in P_{k-1}(\tau)$ for any $\boldsymbol{v} \in V(\tau)$ that satisfies:
	\begin{equation}\label{equ:mweakcurl}
		(\mathrm{curl} _{w, k-1, \tau}\ \boldsymbol{v}, \phi)_\tau:=
		(\boldsymbol{v}^0, \mathbf{curl}\ \phi)_\tau+\langle\Lbrace\boldsymbol{v}^{0}\Rbrace_e\cdot\boldsymbol{t}, \phi\rangle_{\partial \tau}, \forall \phi\in P_{k-1}(\tau).
	\end{equation}
\end{definition}

Without confusion, we simply denote $\mathrm{curl} _{w,k-1,\tau}$ as $\mathrm{curl} _{w}$ in the following part.

\subsection{The modified weak Galerkin method for model problem}

Next, for $k\geqslant 1$, we define MWG finite element spaces as follows:
\begin{eqnarray*}
	V_h:=\left\{\boldsymbol{v}_h
	=
	\left\{\boldsymbol{v}^{0}_h, \boldsymbol{v}^{b}_h\right\} :\left.\boldsymbol{v}_h^{0}\right|_{\tau} \in (P_{k}(\tau))^2,\left.\boldsymbol{v}^{b}_h\right|_{e}
	=\Lbrace\boldsymbol{v}^{0}_h\Rbrace_e, e\subset \partial \tau, \tau \in \mathcal{T}_h\right\},
\end{eqnarray*}
and
\begin{eqnarray*}
	V^0_h:=\left\{\boldsymbol{v}_h=\left\{\boldsymbol{v}^{0}_h, \boldsymbol{v}^{b}_h\right\} :\boldsymbol{v}_h \in V_{h}, \left. \boldsymbol{v}^{b}_h\right|_{e}\cdot\boldsymbol{t}=0, e\in \partial \Omega \right\}.
\end{eqnarray*}

The modified weak Galerkin method for \eqref{Equ:1.1} and \eqref{Equ:1.2} can be obtained by finding $\boldsymbol{u}_h=\left\{\boldsymbol{u}_{h}^0, \Lbrace\boldsymbol{u}_{h}^0\Rbrace_e\right\}\in V^0_h$ which  satisfies
\begin{equation}\label{Eq:MWG}
	a(\boldsymbol{u}_h, \boldsymbol{v}_h)
	=
	(\boldsymbol{f}, \boldsymbol{v}^0_h)_{\mathcal{T}_h}, \quad \forall \boldsymbol{v}_h=\left\{\boldsymbol{v}_h^0, \Lbrace\boldsymbol{v}_h^0\Rbrace_e\right\} \in V^0_h,
\end{equation}
where
\begin{equation*}
	a(\boldsymbol{u}_h, \boldsymbol{v}_h)
	= (\mathrm{curl}_w\ \boldsymbol{u}_h, \mathrm{curl}_w\ \boldsymbol{v}_h)_{\mathcal{T}_h} +(\boldsymbol{u}_h^0, \boldsymbol{v}_h^0)_{\mathcal{T}_h} + s(\boldsymbol{u}_h, \boldsymbol{v}_h),
\end{equation*}
and
\begin{equation}\label{Eqn:2-2-3-b}
	s(\boldsymbol{u}_h,  \boldsymbol{v}_h)=
	\sum\limits_{e\in \mathcal{E}_h}h^{-1}_e \left\langle\Lbracket\boldsymbol{u}_h^0\Rbracket_e, \Lbracket\boldsymbol{v}_h^0\Rbracket_e\right\rangle_e,
\end{equation}
here, we denote $h_{e} = |e|$ the local mesh size of the edge $e$.

\begin{lemma}
	The modified weak Galerkin methods defined in \eqref{Eq:MWG} has a unique solution.
\end{lemma}
\begin{proof}
	Since \eqref{Eq:MWG} is essentially a linear system of equations, we only need to prove the
	uniqueness.  Let $\boldsymbol{f} = 0$. Taking $\boldsymbol{v}_h = \boldsymbol{u}_h$ in \eqref{Eq:MWG}, we obtain
	\begin{equation*}
		(\mathrm{curl}_w\ \boldsymbol{u}_h, \mathrm{curl}_w\ \boldsymbol{u}_h)_{\mathcal{T}_h}
		+(\boldsymbol{u}_h^0, \boldsymbol{u}_h^0)_{\mathcal{T}_h}
		+ \sum\limits_{e\in \mathcal{E}_h}h^{-1}_e\left\langle\Lbracket\boldsymbol{u}_h^0\Rbracket_e, \Lbracket\boldsymbol{u}_h^0\Rbracket_e\right\rangle_e
		=0,
	\end{equation*}
	then we get $\boldsymbol{u}_h^0|_\tau=0, \forall\tau\in \mathcal{T}_h$ and $\Lbracket\boldsymbol{u}_h^0\Rbracket_e=0, \forall e\in \mathcal{E}_h$.
	As a result, $\boldsymbol{u}_h=\left\{\boldsymbol{u}_{h}^0, \Lbrace\boldsymbol{u}_{h}^0\Rbrace_e\right\}=\boldsymbol{0}$.
\end{proof}


\section{A priori error bound}\label{sec:3}
In this section, we will introduce some projections and the corresponding properties, and prove the priori error bound in the energy norm by introducing some technical equations.

\subsection{Some projections and the corresponding approximation}

In order to prove the error estimate, we introduce three projections $P_h, Q_{h}$ and $G_h$ defined as following.
For each element $\tau\in\mathcal{T}_h$ denote by $P_h$ the $L^2$ projections onto $[P_k(\tau)]^2$. The projection $Q_{h}$ is the $L^2$ projection onto $V^0_h$ such that
\begin{equation}\label{Eq:Q_h}
	\left. Q_{h} \boldsymbol{v}\right|_\tau=\left\{P_{h} \boldsymbol{v}, \Lbrace P_{h} \boldsymbol{v}\Rbrace_e\right\}.
\end{equation}
In addition, denote by $G_h$ the $L^2$ projection onto $P_{k-1}(\tau)$.
Next we will present the properties of the projections $P_{h}$, $Q_h$ and $G_h$.

\begin{lemma}[\cite{MuLWangJP15:363}, Lemma 6.1]\label{lem:P_h-Q_h}
	Given a shape-regular triangulation $\mathcal{T}_h$ for $\Omega$, $\boldsymbol{v}\in \boldsymbol{H}^{t+1}(\Omega)$, $0\leqslant t \leqslant k$. Then, for $0\leqslant s \leqslant 1$, we have
	\footnote{Throughout the paper,  we will use $x\lesssim y$ means $x \leq Cy$ where $C$ are generic positive constants independent of the variables that appear in the inequalities and especially the mesh parameters.}
	\begin{eqnarray*}\label{Eqn:2.4.2}
		&
		\sum\limits_{\tau\in \mathcal{T}_{h}}h^{2s}_\tau\left\|\boldsymbol{v}-P_{h}\boldsymbol{v}\right\|^{2}_{s, \tau}\lesssim h^{2(t+1)}\left\|\boldsymbol{v}\right\|^{2}_{t+1},  \\  \label{Eqn:2-6-d}
		&
		\sum\limits_{\tau\in \mathcal{T}_{h}}h^{2s}_{\tau}\left\|\mathrm{curl} \ \boldsymbol{v}-G_h(\mathrm{curl}\ \boldsymbol{v}) \right\|^{2}_{s, \tau} \lesssim h^{2t}\left\|\boldsymbol{v}\right\|^{2}_{t+1}.
	\end{eqnarray*}
\end{lemma}

\begin{lemma}\label{Lem:Q_h}
	Let $\boldsymbol{v}\in 
	\boldsymbol{H}^{t+1}(\Omega),~0\leqslant t\leqslant k$, then
	\begin{equation*}\label{Eq:weakcurlofQ_h}
		\left\|\mathrm{curl}_w\ Q_{h} \boldsymbol{v}-\mathrm{curl}\  \boldsymbol{v}\right\|_{\mathcal{T}_h} \lesssim h^{t} \left\|\boldsymbol{v}\right\|_{t+1}.
	\end{equation*}
\end{lemma}

\begin{proof}
	Let $\tau\in \mathcal{T}_h$, then $\left. Q_{h} \boldsymbol{v}\right|_\tau=\left\{P_{h} \boldsymbol{v}, \Lbrace P_{h} \boldsymbol{v}\Rbrace_e\right\}$. For any $q\in P_{k-1}(\tau)$, using the Definition \ref{def:modifiedweakcurl}, \eqref{Eq:Q_h}, the definition of $P_h$ and Green Formula, we have
	\begin{eqnarray*}
		\lefteqn{\left(\mathrm{curl} _w\ Q_{h} \boldsymbol{v}, q\right)_\tau } \\
		&=&
		\left(P_{h} \boldsymbol{v}, \mathbf{curl}\ q\right)_\tau+\left\langle \Lbrace P_{h} \boldsymbol{v}\Rbrace_e\cdot\boldsymbol{t}, q \right\rangle_{\partial \tau}
		\\
		&=&
		\left(\boldsymbol{v}, \mathbf{curl}\ q\right)_\tau+\left\langle \Lbrace P_{h} \boldsymbol{v}\Rbrace_e\cdot\boldsymbol{t}, q\right\rangle_{\partial \tau}
		\\
		&=&
		\left(\boldsymbol{v}, \mathbf{curl}\ q \right)_\tau
		+\left\langle\boldsymbol{v}\cdot\boldsymbol{t}, q\right\rangle_{\partial \tau}
		-\left\langle\boldsymbol{v}\cdot\boldsymbol{t}, q \right\rangle_{\partial \tau}
		+\left\langle\Lbrace P_{h} \boldsymbol{v}\Rbrace_e\cdot\boldsymbol{t} , q\right\rangle_{\partial \tau}
		\\
		&=&
		\left(\mathrm{curl}\ \boldsymbol{v}, q\right)_\tau
		+ \left\langle\left(\Lbrace P_{h} \boldsymbol{v}\Rbrace_e - \boldsymbol{v}\right)\cdot\boldsymbol{t}, q\right\rangle_{\partial \tau},
	\end{eqnarray*}
	which implies
	\begin{equation}\label{Eqn:Q_h-1}
		\left(\mathrm{curl} _w\ Q_{h} \boldsymbol{v}, q\right)_\tau-\left(\mathrm{curl}\ \boldsymbol{v}, q \right)_\tau
		=\left\langle\left(\Lbrace P_{h} \boldsymbol{v}\Rbrace_e - \boldsymbol{v}\right)\cdot\boldsymbol{t}, q\right\rangle _{\partial \tau}.
	\end{equation}
	
	Applying the definition of $G_{h}$, \eqref{Eqn:Q_h-1}, Cauchy-Schwarz inequality, trace inequality and Lemma \ref{lem:P_h-Q_h}, we have
	\begin{eqnarray}
		\nonumber
		\lefteqn{\left(\mathrm{curl} _w\ Q_{h} \boldsymbol{v}-G_h\left(\mathrm{curl} \  \boldsymbol{v}\right), q \right)_{\mathcal{T}_h}}
		\\ \nonumber
		&=&
		\left(\mathrm{curl} _w\ Q_{h} \boldsymbol{v}-\mathrm{curl}\ \boldsymbol{v}, q\right)_{\mathcal{T}_h}
		\\ \nonumber
		&=&
		\sum\limits_{\tau\in \mathcal{T}_h}\int_{\partial \tau} (\left(\boldsymbol{v}-\Lbrace P_{h} \boldsymbol{v}\Rbrace_e \right)\cdot\boldsymbol{t})\cdot q\mathrm{d}s
		\\  \nonumber
		&\leqslant&
		\sum\limits_{\tau\in \mathcal{T}_h} \left\| \Lbrace P_{h} \boldsymbol{v}\Rbrace_e - \boldsymbol{v}\right\|_{0, \partial \tau} \left\| q\right\|_{0, \partial \tau}
		\\ \nonumber
		&\lesssim&
		\sum\limits_{\tau\in \mathcal{T}_h}h^{-\frac{1}{2}}_\tau\left(\left\| P_{h}\boldsymbol{v}- \boldsymbol{v}\right\|_{0, \tau}
		+ h_\tau\left\|\nabla\left(P_{h}\boldsymbol{v}- \boldsymbol{v}\right)\right\|_{0, \tau}\right)h^{-\frac{1}{2}}_\tau \left\| q\right\|_{0, \tau}    \\  \nonumber
		\\ \label{Eq:Q_h-2}
		&\lesssim&
		h^{t}\left\|\boldsymbol{v}\right\|_{t+1} \left\| q\right\|_{\mathcal{T}_h}.
	\end{eqnarray}
	Let $q=\mathrm{curl} _w\ Q_{h} \boldsymbol{v}-G_h\left(\mathrm{curl}\ \boldsymbol{v}\right)$ in \eqref{Eq:Q_h-2},  we arrive at
	\begin{equation*} \label{Eq:Q_h-3}
		\left\|\mathrm{curl} _w\ Q_{h} \boldsymbol{v}-G_h\left(\mathrm{curl}\  \boldsymbol{v}\right)\right\|_{\mathcal{T}_h}
		\lesssim h^{t}\left\|\boldsymbol{v}\right\|_{t+1}.
	\end{equation*}
	
	Using triangle inequality, the above inequality and Lemma \ref{lem:P_h-Q_h},  we get
	\begin{eqnarray*}
		\lefteqn{\left\|\mathrm{curl} _w\ Q_{h} \boldsymbol{v}-\mathrm{curl}\  \boldsymbol{v}\right\|_{\mathcal{T}_h} } \\
		&=&
		\left\|\mathrm{curl} _w\ Q_{h} \boldsymbol{v}-G_h\left(\mathrm{curl}\  \boldsymbol{v}\right)+G_h\left(\mathrm{curl}\  \boldsymbol{v}\right)-\mathrm{curl}\  \boldsymbol{v}\right\|_{\mathcal{T}_h}
		\\
		&\leqslant &
		\left\|\mathrm{curl} _w\ Q_{h} \boldsymbol{v}-G_h\left(\mathrm{curl}\  \boldsymbol{v}\right)\right\|_{\mathcal{T}_h} + \left\|G_h\left(\mathrm{curl}\  \boldsymbol{v}\right)-\mathrm{curl}\  \boldsymbol{v}\right\|_{\mathcal{T}_h}
		\lesssim h^{t}\left\|\boldsymbol{v}\right\|_{t+1},
	\end{eqnarray*}
	which completes the proof.
\end{proof}

Next, we define the following error norm as
\begin{eqnarray} \label{Eq:H-1}
	\interleave \boldsymbol{v}\interleave^2=
	\left\|\mathrm{curl}_w\ \boldsymbol{v}\right\|^2_{\mathcal{T}_h}
	+\|\boldsymbol{v}^0\|^2_{\mathcal{T}_h}
	+s\left(\boldsymbol{v}, \boldsymbol{v}\right),
\end{eqnarray}
where $\boldsymbol{v} \in V_h\cup(\mathbf{H}^{t+1}(\Omega)), 0\leq t\leq k$.

\begin{lemma}\label{lem:errorofQ_h}
	Let $\boldsymbol{v}\in \boldsymbol{H}^{t+1}(\Omega)$,  $0\leqslant t\leqslant k$. We have
	\begin{equation}\label{Eq:errorofQ_h}
		\interleave \boldsymbol{v}-Q_h\boldsymbol{v}\interleave \lesssim h^{t}\|\boldsymbol{v}\|_{t+1},
	\end{equation}
	where the projection $Q_h$ is defined in \eqref{Eq:Q_h}.
\end{lemma}

\begin{proof}
	For the first and second terms in $ \interleave\boldsymbol{v}-Q_h\boldsymbol{v}\interleave^2$, noting that $\boldsymbol{v}\in \boldsymbol{H}^{t+1}(\Omega)$ and using Lemmas \ref{Lem:Q_h} and \ref{lem:P_h-Q_h}, we have
	\begin{equation} \label{Eq:errorofQ_h1}
		\left\|\mathrm{curl}\ \boldsymbol{v}-\mathrm{curl}_w\  Q_h\boldsymbol{v}\right\|^2_{\mathcal{T}_h} + \left\|\boldsymbol{v}-P_h\boldsymbol{v}\right\|^2_{\mathcal{T}_h}
		\lesssim
		h^{2t}\|\boldsymbol{v}\|^2_{t+1}.
	\end{equation}
	
	For the last term $s(\boldsymbol{v}-Q_h\boldsymbol{v}, \boldsymbol{v}-Q_h\boldsymbol{v})$ in $\interleave\boldsymbol{v}-Q_h\boldsymbol{v}\interleave^2$. Using the definitions of bilinear form $s(\cdot, \cdot)$  and interpolation $Q_h$, trace inequality and Lemma \ref{lem:P_h-Q_h}, we have
	\begin{eqnarray}
		\nonumber
		\lefteqn{s(\boldsymbol{v}-Q_h\boldsymbol{v}, \boldsymbol{v}-Q_h\boldsymbol{v})}
		\\  \nonumber
		&=&
		\sum\limits_{e\in \mathcal{E}_h}h^{-1}_e \left\|\Lbracket \boldsymbol{v}-P_h\boldsymbol{v}\Rbracket _e\right\|^2_{0, e}
		\lesssim
		\sum\limits_{e\in \mathcal{E}_h}h^{-1}_e\left\|\boldsymbol{v}-P_h\boldsymbol{v}\right\|^2_{0, e}
		\\   \nonumber
		&\lesssim&
		\sum\limits_{e\in \mathcal{E}_h}h^{-1}_e\left(h^{-1}_\tau\left\|\boldsymbol{v}-P_h\boldsymbol{v}\right\|^2_{0, \omega_e}
		+ h_\tau \left\|\nabla\left(\boldsymbol{v}-P_h\boldsymbol{v}\right)\right\|^2_{0, \omega_e}\right) \\   \nonumber
		\\ \label{Eq:errorofQ_h2}
		&\lesssim&
		h^{2t}\|\boldsymbol{v}\|^2_{t+1},
	\end{eqnarray}
	where $\omega_e$ is the macro-element associated with the edge $e$ and the constant depends on the shape regularity of $\mathcal{T}_h$.
	
	At last, combining \eqref{Eq:H-1}, \eqref{Eq:errorofQ_h1} with \eqref{Eq:errorofQ_h2}, we have
	\begin{eqnarray*}
		\lefteqn{ \interleave\boldsymbol{v}-Q_h\boldsymbol{v}\interleave^2}  \\
		&=&
		\left\|\mathrm{curl}\ \boldsymbol{v}-\mathrm{curl}_w\ Q_h\boldsymbol{v}\right\|^2_{\mathcal{T}_h} + \left\|\boldsymbol{v}-P_h\boldsymbol{v}\right\|^2_{\mathcal{T}_h}
		+s(\boldsymbol{v}-Q_h\boldsymbol{v}, \boldsymbol{v}-Q_h\boldsymbol{v})
		\\
		&\lesssim&
		h^{2t}\|\boldsymbol{v}\|^2_{t+1},
	\end{eqnarray*}
	which completes the proof.
\end{proof}

\subsection{A priori error estimate}

The variational problem of \eqref{Equ:1.1}-\eqref{Equ:1.2} is to find  $\boldsymbol{u}\in \boldsymbol{H}_0(\mathrm{curl}; \Omega)$,  such that
\begin{equation}\label{CVP}
	(\mathrm{curl}\ \boldsymbol{u}, \mathrm{curl}\ \boldsymbol{v}) + (\boldsymbol{u}, \boldsymbol{v})=(\boldsymbol{f}, \boldsymbol{v}),
	\quad \forall \boldsymbol{v}\in \boldsymbol{H}_0(\mathrm{curl}; \Omega).
\end{equation}

Before we provide the a priori error estimate, we need to estimate $\interleave Q_h\boldsymbol{u}-\boldsymbol{u}_h\interleave$ for weak solution $\boldsymbol{u}$ of \eqref{CVP} and modified WG approximation  $\boldsymbol{u}_h$ of \eqref{Eq:MWG}. Hence, we introduce two technique equations, which can be found in  \eqref{Eq:compatibility-1} and \eqref{Eqn:2.16}, respectively.


\begin{lemma}\label{Lem:compatibility}
	Let $\boldsymbol{u}\in \boldsymbol{H} _0(\mathrm{curl}; \Omega)\cap \boldsymbol{H}^{t+1}(\Omega)$ is the solution of \eqref{CVP}, where $1/2 \langle t\leqslant k$. Then for any $\boldsymbol{v}_h=\left\{\boldsymbol{v}^0_h, \Lbrace\boldsymbol{v}^0_h\Rbrace_e\right\}\in \boldsymbol{V}^0_h$, we have
	\begin{equation}\label{Eq:compatibility-1}
		\left(\mathrm{curl}\ \boldsymbol{u}, \mathrm{curl}_w\ \boldsymbol{v}_h\right)_{\mathcal{T}_h}+\left(\boldsymbol{u}, \boldsymbol{v}^0_h \right)_{\mathcal{T}_h}-\left(\boldsymbol{f}, \boldsymbol{v}^0_h \right)_{\mathcal{T}_h}-E(\boldsymbol{u}, \boldsymbol{v}_h)
		=0,
	\end{equation}
	where
	\begin{equation}\label{Eq:E}
		E(\boldsymbol{u}, \boldsymbol{v}_h)
		=\sum\limits_{\tau\in \mathcal{T}_h}\left\langle G_h(\mathrm{curl}\ \boldsymbol{u})-\mathrm{curl}\  \boldsymbol{u},  (\Lbrace\boldsymbol{v}^0_h\Rbrace_e-\boldsymbol{v}^0_h)\cdot\boldsymbol{t}\right\rangle_{\partial \tau}.
	\end{equation}
\end{lemma}

\begin{proof}
	Note that $G_h$ is $L^2$ projection, using the definition of discrete modified weak curl \eqref{equ:mweakcurl} and Green formula leads to
	\begin{eqnarray}
		\nonumber
		\lefteqn{\left(\mathrm{curl}\ \boldsymbol{u}, \mathrm{curl}_w\ \boldsymbol{v}_h\right)_\tau
			=
			\left(G_h(\mathrm{curl}\ \boldsymbol{u}), \mathrm{curl}_w\ \boldsymbol{v}_h\right)_\tau
		}
		\\  \nonumber
		&=&
		\left(\mathbf{curl} (G_h(\mathrm{curl}\ \boldsymbol{u})), \boldsymbol{v}^0_h \right)_\tau+\left\langle G_h(\mathrm{curl}\ \boldsymbol{u}), \Lbrace\boldsymbol{v}^0_h\Rbrace_e\cdot\boldsymbol{t} \right\rangle_{\partial \tau}
		\\ \nonumber
		&=&
		\left(G_h(\mathrm{curl}\ \boldsymbol{u}), \mathrm{curl}\ \boldsymbol{v}^0_h \right)_\tau-\left\langle G_h(\mathrm{curl}\ \boldsymbol{u}), \boldsymbol{v}^0_h\cdot\boldsymbol{t}\right\rangle_{\partial \tau}  \\ \nonumber
		&&\ \ +\left\langle G_h(\mathrm{curl}\ \boldsymbol{u}),  \Lbrace\boldsymbol{v}^0_h\Rbrace_e\cdot\boldsymbol{t}\right\rangle_{\partial \tau}   \\ \nonumber
		&=&
		\left(\mathrm{curl}\ \boldsymbol{u}, \mathrm{curl}\ \boldsymbol{v}^0_h \right)_\tau
		+\left\langle G_h(\mathrm{curl}\ \boldsymbol{u}),  \left(\Lbrace\boldsymbol{v}^0_h\Rbrace_e-\boldsymbol{v}^0_h\right)\cdot\boldsymbol{t}\right\rangle_{\partial \tau}
		\\ \nonumber
		&=&
		\left(\mathbf{curl}\mathrm{curl}\ \boldsymbol{u}, \boldsymbol{v}^0_h \right)_\tau
		+\left\langle \mathrm{curl}\ \boldsymbol{u},  \boldsymbol{v}^0_h\cdot\boldsymbol{t}\right\rangle_{\partial \tau}\\ \label{Eqn:2.4.10}
		&& \  +\left\langle G_h(\mathrm{curl}\ \boldsymbol{u}),  \left(\Lbrace\boldsymbol{v}^0_h\Rbrace_e-\boldsymbol{v}^0_h\right)\cdot\boldsymbol{t}\right\rangle_{\partial \tau}.
	\end{eqnarray}
	
	Note that $\boldsymbol{u}\in \boldsymbol{H}^{t+1}(\Omega)$ and $(\mathrm{curl}\ \boldsymbol{u})|_{\partial \tau}$ is continuous on $\tau$, we know
	\begin{equation}\label{Eqn:2.4.11}
		\sum\limits_{\tau\in \mathcal{T}_h}\left\langle \mathrm{curl}\ \boldsymbol{u}, \Lbrace\boldsymbol{v}^0_h \Rbrace_e\cdot\boldsymbol{t}\right\rangle_{\partial \tau}
		=\sum\limits_{e\in \mathcal{E}_h}\left\langle \Lbracket\mathrm{curl}\ \boldsymbol{u}\Rbracket_e,   \Lbrace\boldsymbol{v}^0_h\Rbrace_e \right\rangle_e
		=0, 
	\end{equation}
where the tangential jump $\Lbracket\mathrm{curl} \boldsymbol{u}\Rbracket_e=$	
	
	Adding $\left(\boldsymbol{u}, \boldsymbol{v}^0_h \right)_{\mathcal{T}_h}$ in both sides of \eqref{Eqn:2.4.10}, using \eqref{Eqn:2.4.11} and \eqref{CVP}, we have
	\begin{eqnarray*}
		\lefteqn{\left(\mathrm{curl}\ \boldsymbol{u}, \mathrm{curl}_w\  \boldsymbol{v}_h\right)_{\mathcal{T}_h}+\left(\boldsymbol{u}, \boldsymbol{v}^0_h \right)_{\mathcal{T}_h}=
			\sum\limits_{\tau\in \mathcal{T}_h} \left(\left(\mathrm{curl}\ \boldsymbol{u}, \mathrm{curl}_w\ \boldsymbol{v}_h\right)_\tau+\left(\boldsymbol{u}, \boldsymbol{v}^0_h \right)_\tau\right)}
		\\
		&=&
		\sum\limits_{\tau\in \mathcal{T}_h}\Big( \left(\mathbf{curl}\mathrm{curl}\ \boldsymbol{u}, \boldsymbol{v}^0_h \right)_\tau +\left(\boldsymbol{u}, \boldsymbol{v}^0_h \right)_\tau
		+\left\langle \mathrm{curl}\ \boldsymbol{u}, \boldsymbol{v}^0_h\cdot\boldsymbol{t}\right\rangle_{\partial \tau} \\
		&&\ \ +\left\langle G_h(\mathrm{curl}\ \boldsymbol{u}),  \left(\Lbrace\boldsymbol{v}^0_h\Rbrace_e-\boldsymbol{v}^0_h\right)\cdot\boldsymbol{t}\right\rangle_{\partial \tau} \Big) \\
		&=&
		\sum\limits_{\tau\in \mathcal{T}_h} \left(\left(\mathbf{curl} \mathrm{curl}\ \boldsymbol{u}+\boldsymbol{u}, \boldsymbol{v}^0_h \right)_\tau -\left\langle \mathrm{curl}\ \boldsymbol{u},  (\Lbrace\boldsymbol{v}^0_h\Rbrace_e-\boldsymbol{v}^0_h)\cdot\boldsymbol{t}\right\rangle_{\partial \tau} \right. \\
		&& \ \ + \left. \left\langle G_h(\mathrm{curl}\ \boldsymbol{u}),  \left(\Lbrace\boldsymbol{v}^0_h\Rbrace_e-\boldsymbol{v}^0_h\right)\cdot\boldsymbol{t}\right\rangle_{\partial \tau}  \right)
		\\
		&=&
		\left(\boldsymbol{f}, \boldsymbol{v}^0_h \right)_{\mathcal{T}_h}+\sum\limits_{\tau\in \mathcal{T}_h}\left\langle G_h(\mathrm{curl}\ \boldsymbol{u})-\mathrm{curl}\ \boldsymbol{u},  (\Lbrace\boldsymbol{v}^0_h\Rbrace_e-\boldsymbol{v}^0_h)\cdot\boldsymbol{t}\right\rangle_{\partial \tau},
	\end{eqnarray*}
	Then we complete the proof by making a definition
	$$
	E(\boldsymbol{u}, \boldsymbol{v}_h)=\sum\limits_{\tau\in \mathcal{T}_h}\left\langle G_h(\mathrm{curl}\ \boldsymbol{u})-\mathrm{curl}\ \boldsymbol{u},  (\Lbrace\boldsymbol{v}^0_h\Rbrace_e-\boldsymbol{v}^0_h)\cdot\boldsymbol{t}\right\rangle_{\partial \tau}.
	$$
\end{proof}

The following Lemma is based on the average and tangential jump of vector function \eqref{vectorinterior}, \eqref{vectorboundary} and the above definitions of the average and tangential jump of scalar function \eqref{average}, \eqref{jump}.
\begin{lemma}
	For any $\phi\in H^s\left(\Omega\right)$ and $\boldsymbol{v}\in \boldsymbol{H}^s\left(\Omega\right)$, $s>1/2$, we have
	\begin{equation}\label{Eqn:2.16}
		\sum\limits_{\tau\in \mathcal{T}_h} \langle (\Lbrace\boldsymbol{v}\Rbrace_e-\boldsymbol{v})\cdot\boldsymbol{t}, \phi\rangle_{\partial \tau}
		= -\sum\limits_{e\in \mathcal{E}^0_h}\langle\Lbrace\phi\Rbrace_e, \Lbracket \boldsymbol{v}\Rbracket_e \rangle_e.
	\end{equation}
	
\end{lemma}

Based on the two technique equations \eqref{Eq:compatibility-1} and \eqref{Eqn:2.16},  it is turn to estimate $\interleave Q_h\boldsymbol{u}-\boldsymbol{u}_h\interleave$
for weak solution $\boldsymbol{u}$ of \eqref{CVP} and modified WG approximation  $\boldsymbol{u}_h$ of \eqref{Eq:MWG}.
\begin{lemma}\label{lem:3.3.5}
	Assume that  $\boldsymbol{u}\in \boldsymbol{H} _0(\mathrm{curl}; \Omega)\cap \boldsymbol{H}^{t+1}(\Omega)$ with $1/2\langle t \leqslant k$   and $\boldsymbol{u}_h\in \boldsymbol{V}^0_h$ are the solutions of \eqref{CVP} and \eqref{Eq:MWG}, respectively. If $h<1$,  then we have \begin{equation*}\label{Eqn:2.4.19}
		\interleave Q_h\boldsymbol{u}-\boldsymbol{u}_h\interleave\lesssim h^t\|\boldsymbol{u}\|_{t+1},
	\end{equation*}
	where the constant depends on the shape regularity of $\mathcal{T}_h$.
\end{lemma}

\begin{proof}
	For any $\boldsymbol{v}_h=\left\{\boldsymbol{v}^0_h, \Lbrace\boldsymbol{v}^0_h\Rbrace_e\right\} \in \boldsymbol{V}^0_h$, using \eqref{Eq:compatibility-1} and \eqref{Eq:MWG}, we have
	\begin{eqnarray}\nonumber
		\lefteqn{\left(\mathrm{curl}_w\ Q_h\boldsymbol{u}, \mathrm{curl}_w\ \boldsymbol{v}_h\right)_{\mathcal{T}_h}+\left(P_h\boldsymbol{u}, \boldsymbol{v}^0_h\right)_{\mathcal{T}_h} } \\  \nonumber
		&=&
		\left(\mathrm{curl}_w\ Q_h\boldsymbol{u}, \mathrm{curl}_w\ \boldsymbol{v}_h\right)_{\mathcal{T}_h}+\left(P_h\boldsymbol{u}, \boldsymbol{v}^0_h\right)_{\mathcal{T}_h} \\ \nonumber
		&& \ \ +\Big(\left(\boldsymbol{f}, \boldsymbol{v}^0_h \right)_{\mathcal{T}_h} +E(\boldsymbol{u}, \boldsymbol{v}_h) -\left(\mathrm{curl}\ \boldsymbol{u}, \mathrm{curl}_w\ \boldsymbol{v}_h\right)_{\mathcal{T}_h}-\left(\boldsymbol{u}, \boldsymbol{v}^0_h \right)_{\mathcal{T}_h}\Big)  \\  \nonumber
		&=&
		\left(\boldsymbol{f}, \boldsymbol{v}^0_h \right)_{\mathcal{T}_h} +E(\boldsymbol{u}, \boldsymbol{v}_h)+\left(\mathrm{curl}_w\ Q_h\boldsymbol{u}-\mathrm{curl}\ \boldsymbol{u}, \mathrm{curl}_w\ \boldsymbol{v}_h\right)_{\mathcal{T}_h} +\left(P_h\boldsymbol{u}-\boldsymbol{u}, \boldsymbol{v}^0_h\right)_{\mathcal{T}_h}
		\\  \nonumber
		&=&
		(\mathrm{curl}_w\ \boldsymbol{u}_h, \mathrm{curl}_w\ \boldsymbol{v})_{\mathcal{T}_h}
		+(\boldsymbol{u}_h^0, \boldsymbol{v}^0_h)_{\mathcal{T}_h}  +s(\boldsymbol{u}_h,  \boldsymbol{v}_h)
		\\  \label{Eqn:2-6-m}
		&&
		+E(\boldsymbol{u}, \boldsymbol{v}_h) +\left(\mathrm{curl}_w\ Q_h\boldsymbol{u}-\mathrm{curl}\ \boldsymbol{u}, \mathrm{curl}_w\ \boldsymbol{v}_h\right)_{\mathcal{T}_h} +\left(P_h\boldsymbol{u}-\boldsymbol{u}, \boldsymbol{v}^0_h\right)_{\mathcal{T}_h}.
	\end{eqnarray}
	Adding $s(Q_h\boldsymbol{u},  \boldsymbol{v}_h)$ in the both sides of \eqref{Eqn:2-6-m}, we obtain
	\begin{eqnarray*}
		\lefteqn{\left(\mathrm{curl}_w\ Q_h\boldsymbol{u}, \mathrm{curl}_w\ \boldsymbol{v}_h\right)_{\mathcal{T}_h}+\left(P_h\boldsymbol{u}, \boldsymbol{v}^0_h\right)_{\mathcal{T}_h} +s(Q_h\boldsymbol{u},  \boldsymbol{v}_h)}\\
		&=&(\mathrm{curl}_w\ \boldsymbol{u}_h, \mathrm{curl}_w\ \boldsymbol{v})_{\mathcal{T}_h}+(\boldsymbol{u}_h^0, \boldsymbol{v}^0_h)_{\mathcal{T}_h}+s(\boldsymbol{u}_h,  \boldsymbol{v}_h) \\
		&&+E(\boldsymbol{u}, \boldsymbol{v}_h)+\left(\mathrm{curl}_w\ Q_h\boldsymbol{u}-\mathrm{curl}\ \boldsymbol{u}, \mathrm{curl}_w\ \boldsymbol{v}_h\right)_{\mathcal{T}_h} \\ &&+\left(P_h\boldsymbol{u}-\boldsymbol{u}, \boldsymbol{v}^0_h\right)_{\mathcal{T}_h}+s(Q_h\boldsymbol{u},  \boldsymbol{v}_h),
	\end{eqnarray*}
	which can be rewritten as
	\begin{eqnarray}\nonumber
\lefteqn{(\mathrm{curl}_w (Q_h\boldsymbol{u}-\boldsymbol{u}_h), \mathrm{curl}_w\ \boldsymbol{v}_h)_{\mathcal{T}_h}+\left(P_h\boldsymbol{u}-\boldsymbol{u}_h^0, \boldsymbol{v}^0_h\right)_{\mathcal{T}_h}+ s(Q_h\boldsymbol{u}-\boldsymbol{u}_h,  \boldsymbol{v}_h)}\\\label{Eqn:2.4.14}
&&=E(\boldsymbol{u}, \boldsymbol{v}_h)+\left(\mathrm{curl}_w\ Q_h\boldsymbol{u}-\mathrm{curl} \boldsymbol{u}, \mathrm{curl}_w \boldsymbol{v}_h\right)_{\mathcal{T}_h} \nonumber \\ \label{Eqn:2.4.14}
&&+\left(P_h\boldsymbol{u}-\boldsymbol{u}, \boldsymbol{v}^0_h\right)_{\mathcal{T}_h}+s(Q_h\boldsymbol{u},  \boldsymbol{v}_h).
\end{eqnarray}
	Next, we will estimate each of these terms in the right hand side of \eqref{Eqn:2.4.14}.
	
	For the first term, using \eqref{Eq:E}, \eqref{Eqn:2.16}, Cauchy-Schwarz inequality, trace inequality and Lemma \ref{lem:P_h-Q_h}, we obtain
	\begin{eqnarray}
		\nonumber
		\lefteqn{E(\boldsymbol{u}, \boldsymbol{v}_h)
			=
			\sum\limits_{\tau\in \mathcal{T}_h}\left\langle G_h(\mathrm{curl}\ \boldsymbol{u})-\mathrm{curl}\ \boldsymbol{u}, \left(\Lbrace\boldsymbol{v}^0_h\Rbrace_e-\boldsymbol{v}^0_h\right)\cdot\boldsymbol{t} \right\rangle_{\partial \tau}}  \\ \nonumber
		&=&
		-\sum\limits_{e\in \mathcal{E}_h}\left\langle \Lbrace G_h(\mathrm{curl}\ \boldsymbol{u})-\mathrm{curl}\ \boldsymbol{u}\Rbrace_e, \Lbracket \boldsymbol{v}^0_h\Rbracket _e\right\rangle_e
		\\  \nonumber
		&=&
		-\sum\limits_{e\in \mathcal{E}_h}\left\langle h^{\frac{1}{2}}_e \Lbrace G_h(\mathrm{curl}\ \boldsymbol{u})-\mathrm{curl}\ \boldsymbol{u}\Rbrace_e, h^{-\frac{1}{2}}_e \Lbracket \boldsymbol{v}^0_h\Rbracket _e\right\rangle_e
		\\ \nonumber
		&\lesssim&
		\left(\sum\limits_{e\in \mathcal{E}_h}\int_e h_e\left|G_h(\mathrm{curl}\ \boldsymbol{u})-\mathrm{curl}\ \boldsymbol{u}\right|^2ds\right)^{\frac{1}{2}}
		\left(\sum\limits_{e\in \mathcal{E}_h}\int_e h^{-1}_e \left|\Lbracket \boldsymbol{v}^0_h\Rbracket _e\right|^2ds\right)^{\frac{1}{2}}  \\  \nonumber
		&\lesssim&
		\sum\limits_{e\in \mathcal{E}_h}h^{\frac{1}{2}}_e \left\|G_h(\mathrm{curl}\ \boldsymbol{u})-\mathrm{curl}\ \boldsymbol{u}\right\|_{0, e} \sqrt{s(\boldsymbol{v}_h, \boldsymbol{v}_h)}
		\\  \nonumber
		&\lesssim&
		\sum\limits_{e\in \mathcal{E}_h}h^{\frac{1}{2}}_e h^{-\frac{1}{2}}_\tau \left(\left\|G_h(\mathrm{curl}\ \boldsymbol{u})-\mathrm{curl}\ \boldsymbol{u}\right\|_{0, \omega_e} \right. \\  \nonumber
		&&\ \ \left. +h_\tau\left\|\nabla\left(G_h(\mathrm{curl}\ \boldsymbol{u})-\mathrm{curl}\ \boldsymbol{u}\right)\right\|_{0, \omega_e}\right)\sqrt{s(\boldsymbol{v}_h,  \boldsymbol{v}_h)}
		\\  \label{Eqn:2.4.16}
		&\lesssim&
		h^t\|\boldsymbol{u}\|_{t+1}\sqrt{s(\boldsymbol{v}_h,  \boldsymbol{v}_h)}.
	\end{eqnarray}
	
	For the second term, combining Cauchy-Schwarz inequality and Lemma \ref{lem:errorofQ_h}, we arrive at
	\begin{eqnarray}\nonumber
		\left(\mathrm{curl}_w Q_h\boldsymbol{u}-\mathrm{curl} \boldsymbol{u}, \mathrm{curl}_w \boldsymbol{v}_h\right)_{\mathcal{T}_h}
		&\leqslant&\left\|\mathrm{curl}_w Q_h\boldsymbol{u}-\mathrm{curl} \boldsymbol{u}\right\|_{\mathcal{T}_h} \left\|\mathrm{curl}_w \boldsymbol{v}_h\right\|_{\mathcal{T}_h} \\ \label{Eqn:2.4.17}
		&\lesssim&
		h^t\|\boldsymbol{u}\|_{t+1}\left\|\mathrm{curl}_w\ \boldsymbol{v}_h\right\|_{\mathcal{T}_h}.
	\end{eqnarray}
	
	For the third term, using Cauchy-Schwarz inequality and Lemma \ref{lem:P_h-Q_h} leads to
	\begin{eqnarray}
		\left(P_h\boldsymbol{u}-\boldsymbol{u}, \boldsymbol{v}^0_h\right)_{\mathcal{T}_h}
		\lesssim
		\left\|P_h\boldsymbol{u}-\boldsymbol{u}\right\|_{\mathcal{T}_h} \left\|\boldsymbol{v}^0_h\right\|_{\mathcal{T}_h}
		\lesssim
		h^{t+1}\|\boldsymbol{u}\|_{t+1} \left\|\boldsymbol{v}^0_h\right\|_{\mathcal{T}_h}.\label{Eqn:2.4.18}
	\end{eqnarray}
	
	For the fourth term,  applying the definition of $s(\cdot,  \cdot)$ in \eqref{Eqn:2-2-3-b}, the fact $\Lbracket \boldsymbol{u}\Rbracket_e=0(\forall e\in \mathcal{E}_h)$ for $\boldsymbol{u}\in  \boldsymbol{H} _0(\mathrm{curl}; \Omega)$,  Cauchy-Schwarz inequality, trace inequality and Lemma \ref{lem:P_h-Q_h}, we have
	\begin{eqnarray}
		\nonumber
		s(Q_h\boldsymbol{u}, \boldsymbol{v}_h)
		&=&\sum\limits_{e\in \mathcal{E}_h}h^{-1}_{e}\left\langle \Lbracket P_h\boldsymbol{u}\Rbracket _e,   \Lbracket \boldsymbol{v}^0_h\Rbracket_e  \right\rangle_e
		=
		\sum\limits_{e\in \mathcal{E}_h}h^{-1}_e\left\langle \Lbracket P_h\boldsymbol{u}-\boldsymbol{u}\Rbracket_e, \Lbracket \boldsymbol{v}^0_h\Rbracket_e  \right\rangle_e
		\\  \nonumber
		&=&
		\sum\limits_{e\in \mathcal{E}_h}\left\langle h^{-1/2}_e  \Lbracket P_h\boldsymbol{u}-\boldsymbol{u}\Rbracket_e,  h^{-1/2}_e \Lbracket \boldsymbol{v}^0_h\Rbracket _e \right\rangle_e
		\\  \nonumber
		&\leqslant&
		\left(\sum\limits_{e\in \mathcal{E}_h}\int_e h^{-1}_e \left| \Lbracket P_h\boldsymbol{u}-\boldsymbol{u}\Rbracket _e\right|^2ds\right)^{\frac{1}{2}}
		 \left(\sum\limits_{e\in \mathcal{E}_h}\int_e h^{-1}_e \left|\Lbracket \boldsymbol{v}^0_h\Rbracket _e\right|^2ds\right)^{\frac{1}{2}}  \\  \nonumber
		&\lesssim&
		\sum\limits_{e\in \mathcal{E}_h}h^{-\frac{1}{2}}_e \left\|P_h\boldsymbol{u}-\boldsymbol{u}\right\|_{0, e} \sqrt{s(\boldsymbol{v}_h,  \boldsymbol{v}_h)}
		\\  \nonumber
		&\lesssim&
		\sum\limits_{e\in \mathcal{E}_h}h^{-\frac{1}{2}}_e h^{-\frac{1}{2}}_\tau\left(\left\|P_h\boldsymbol{u}-\boldsymbol{u}\right\|_{0, \omega_e}+h_\tau\left\|\nabla\left(P_h\boldsymbol{u}-\boldsymbol{u}\right)\right\|_{0, \omega_e}\right)\sqrt{s(\boldsymbol{v}_h,  \boldsymbol{v}_h)}  \\  \nonumber
		\\ \label{Eqn:2.4.15}
		&\lesssim&
		h^t\|\boldsymbol{u}\|_{t+1}\sqrt{s(\boldsymbol{v}_h,  \boldsymbol{v}_h)}.
	\end{eqnarray}

	Submitting \eqref{Eqn:2.4.16}-\eqref{Eqn:2.4.15} into \eqref{Eqn:2.4.14}, choosing $\boldsymbol{v}_h=Q_h\boldsymbol{u}-\boldsymbol{u}_h=\left\{P_h\boldsymbol{u}-\boldsymbol{u}_h^0, \Lbrace P_{h} \boldsymbol{u}-\boldsymbol{u}_h^0\Rbrace\right\}$ and using the definition of energy norm \eqref{Eq:H-1}, we have
	\begin{eqnarray*}
		\lefteqn{\interleave Q_h\boldsymbol{u}-\boldsymbol{u}_h\interleave^2 =
			\left(\mathrm{curl}_w\ (Q_h\boldsymbol{u}-\boldsymbol{u}_h), \mathrm{curl}_w\ (Q_h\boldsymbol{u}-\boldsymbol{u}_h)\right)_{\mathcal{T}_h}} 
		\\
		&&
		+\left(P_h\boldsymbol{u}-\boldsymbol{u}_h^0, P_h\boldsymbol{u}-\boldsymbol{u}_h^0\right)_{\mathcal{T}_h} +s(Q_h\boldsymbol{u}-\boldsymbol{u}_h, Q_h\boldsymbol{u}-\boldsymbol{u}_h)  \\
		&\lesssim& h^t\|\boldsymbol{u}\|_{t+1}\sqrt{s (Q_h\boldsymbol{u}-\boldsymbol{u}_h,  Q_h\boldsymbol{u}-\boldsymbol{u}_h)}
		+ h^t\|\boldsymbol{u}\|_{t+1}\left\|\mathrm{curl}_w\ Q_h\boldsymbol{u}
		-\boldsymbol{u}_h\right\|_{\mathcal{T}_h}\\
		&& +h^{t+1}\|\boldsymbol{u}\|_{t+1} \left\|P_hu-u^0_h\right\|_{\mathcal{T}_h}
		+h^t\|\boldsymbol{u}\|_{t+1}\sqrt{s  (Q_h\boldsymbol{u}-\boldsymbol{u}_h, Q_h\boldsymbol{u}-\boldsymbol{u}_h)}   \\
		&\lesssim&
		h^t\|\boldsymbol{u}\|_{t+1}\interleave Q_h\boldsymbol{u}-\boldsymbol{u}_h\interleave,
	\end{eqnarray*}
	in the last inequality, we use the fact $h<1$.
\end{proof}

\begin{theorem}\label{Thm:2.2}
	Assume $\boldsymbol{u}\in \boldsymbol{H} _0(\mathrm{curl}; \Omega)\cap \boldsymbol{H}^{t+1}(\Omega)$ with $1/2\langle t \leqslant k$ and $\boldsymbol{u}_h\in \boldsymbol{V}^0_h$ are the solutions of \eqref{CVP} and \eqref{Eq:MWG}, respectively. If $h<1$, then we have
	\begin{equation}\label{Eqn:2.4.12}
		\interleave \boldsymbol{u}-\boldsymbol{u}_h\interleave\lesssim h^t\left\|\boldsymbol{u}\right\|_{t+1},
	\end{equation}
	where the constant depends on the shape regularity of $\mathcal{T}_h$.
\end{theorem}
\begin{proof}
	Applying definition of energy norm \eqref{Eq:H-1}, triangle inequality, Lemmas \ref{Lem:compatibility} and \ref{lem:3.3.5}, we have
	\begin{eqnarray*}
		\interleave\boldsymbol{u}-\boldsymbol{u}_h\interleave &=&
		\interleave\boldsymbol{u}-Q_h\boldsymbol{u}+Q_h\boldsymbol{u}-\boldsymbol{u}_h\interleave \\
		&\leqslant&
		\interleave\boldsymbol{u}-Q_h\boldsymbol{u}\interleave+\interleave Q_h\boldsymbol{u}-\boldsymbol{u}_h\interleave
		\lesssim
		h^t\|\boldsymbol{u}\|_{t+1},
	\end{eqnarray*}
	which completes the proof.
\end{proof}

\section{Numerical experiment}\label{sec:5}
In this section, we will implement some numerical experiments to verify the convergence rate of  modified weak Galerkin algorithm \eqref{Eq:MWG}.
In all experiments, our computational domain is $\Omega=[0,1]\times [0,1]$.  $\Omega$ is partitioned by a uniform square meshes, then each mesh element is obtained by breaking up every square element into two triangles. Let $h=1/N$ be mesh sizes for different triangular meshes, where $N$ means to divide $x-$ and $y-$ into $N$ uniformly distributed subintervals.

\begin{example}\label{exa:1}
	In this example, we also choose the linear MWG finite element space 
	\begin{equation}\label{linearspace}
	V_h:=\left\{\boldsymbol{v}_h
	=
	\left\{\boldsymbol{v}^{0}_h, \boldsymbol{v}^{b}_h\right\} :\left.\boldsymbol{v}_h^{0}\right|_{\tau} \in (P_{1}(\tau))^2,\left.\boldsymbol{v}^{b}_h\right|_{e}
	=\Lbrace\boldsymbol{v}^{0}_h\Rbrace_e, e\subset \partial \tau, \tau \in \mathcal{T}_h\right\},
	\end{equation} and the second order MWG finite element space
$$
	V_h=\left\{\boldsymbol{v}_h
	=
	\left\{\boldsymbol{v}^{0}_h, \boldsymbol{v}^{b}_h\right\} :\left.\boldsymbol{v}_h^{0}\right|_{\tau} \in (P_{2}(\tau))^2,\left.\boldsymbol{v}^{b}_h\right|_{e}
	=\Lbrace\boldsymbol{v}^{0}_h\Rbrace_e, e\subset \partial \tau, \tau \in \mathcal{T}_h\right\}.
$$
The true solution is equal to the vector function: $\boldsymbol{u}=(x(1-x)y(1-y),x(1-x)y(1-y))^T$.\end{example}
	We report the errors of Example \ref{exa:1} with the polynomial degree $k=1$ and $k=2$ in Tables \ref{Tab:1} and \ref{Tab:2}, respectively. We can see that $\interleave\boldsymbol{u} -\boldsymbol{u}_h\interleave$ has first order in Table \ref{Tab:1} and has second order in 

\begin{minipage}[t]{0.5\linewidth} 
	\begin{table}[H]
		\centering\caption{ Convergence for Example \ref{exa:1} with $k=1$.}\label{Tab:1}
		\vskip 0.1cm
		\begin{tabular}{{p{1cm}p{2cm}p{2cm}}}\hline
			\multirow{2}{*} {$h$}& \multicolumn{2}{c}{$\interleave\boldsymbol{u}-\boldsymbol{u}_{h}\interleave$}
			\\\cline { 2 - 3 }
			& \multicolumn{1}{c}{Error}  & \multicolumn{1}{c}{order}  \\ \hline
			1/4 & \multicolumn{1}{c}{5.98E-02} &  \multicolumn{1}{c}{N/A} \\
			1/8 & \multicolumn{1}{c}{3.59E-02} & \multicolumn{1}{c}{0.73}
			\\
			1/16 & \multicolumn{1}{c}{1.94E-02} & \multicolumn{1}{c}{0.89}
			\\
			1/32 & \multicolumn{1}{c}{1.01E-02} & \multicolumn{1}{c}{0.95}
			\\
			1/64 & \multicolumn{1}{c}{5.12E-03} & \multicolumn{1}{c}{0.97} \\
			1/128 & \multicolumn{1}{c}{2.58E-03} & \multicolumn{1}{c}{0.99} 			\\
			\hline
		\end{tabular}
	\end{table}
\end{minipage}
\begin{minipage}[t]{0.5\linewidth}
	\begin{table}[H]
		\centering\caption{Convergence for Example \ref{exa:1} with $k=2$.}\label{Tab:2}
		\vskip 0.1cm
		\begin{tabular}{{p{1cm}p{2cm}p{2cm}}}\hline
			\multirow{2}{*} {$h$}& \multicolumn{2}{c}{$\interleave\boldsymbol{u}-\boldsymbol{u}_{h}\interleave$}
			\\\cline { 2 - 3 }
			& \multicolumn{1}{c}{Error}  & \multicolumn{1}{c}{order}   \\ \hline
			1/4 & \multicolumn{1}{c}{1.26E-02} &  \multicolumn{1}{c}{N/A} \\
			1/8 & \multicolumn{1}{c}{3.40E-03} & \multicolumn{1}{c}{1.88}
			\\
			1/16 & \multicolumn{1}{c}{8.85E-04} & \multicolumn{1}{c}{1.94}
			\\
			1/32 & \multicolumn{1}{c}{2.26E-04} & \multicolumn{1}{c}{1.97}
			\\
			1/64 & \multicolumn{1}{c}{5.70E-05} & \multicolumn{1}{c}{1.98} \\
			1/128 & \multicolumn{1}{c}{1.43E-05} & \multicolumn{1}{c}{1.99}    \\
			\hline
		\end{tabular}
	\end{table}
\end{minipage}

\begin{example}\label{exa:2}
	In this example, we choose linear MWG finite element space \eqref{linearspace} and another true solution $\boldsymbol{u}=(e^{x-y} xy\left(1-x\right)\left(1-y\right),\sin \left(\pi x\right) \sin \left(\pi y\right))^T$. 
	\end{example}
	We report the errors of Example \ref{exa:2} in Table \ref{Tab:3} by  fixing different mesh sizes $h$ to test the numerical example. And we observe that $\interleave\boldsymbol{u} -\boldsymbol{u}_h\interleave$ is first order. We obtain the same convergent order, which is in support of \eqref{Eqn:2.4.12} in Theorem \ref{Thm:2.2}, just as in Table \ref{Tab:1}.

\begin{table}[H]
		\centering\caption{Convergence for Example \ref{exa:2} with homogeneous boundary condition}\label{Tab:3}.
		\setlength{\tabcolsep}{12mm}{
		\begin{tabular}{{p{1cm}p{2cm}p{2cm}}}\hline
			\multirow{2}{*} {$h$}& \multicolumn{2}{c}{$\interleave\boldsymbol{u}-\boldsymbol{u}_{h}\interleave$}
			\\\cline { 2 - 3 }
			& \multicolumn{1}{c}{Error}  & \multicolumn{1}{c}{order}   \\ \hline
			1/4 & \multicolumn{1}{c}{7.15E-01} &  \multicolumn{1}{c}{N/A}  \\
			1/8  & \multicolumn{1}{c}{4.22E-01}  & \multicolumn{1}{c}{0.76}
			\\
			1/16& \multicolumn{1}{c}{2.26E-01} & \multicolumn{1}{c}{0.90}  \\
			1/32& \multicolumn{1}{c}{1.170E-01} & \multicolumn{1}{c}{0.95}  \\
			1/64& \multicolumn{1}{c}{5.94E-02} & \multicolumn{1}{c}{0.98}  \\
			1/128& \multicolumn{1}{c}{2.99E-02} & \multicolumn{1}{c}{0.98}  \\
			1/256& \multicolumn{1}{c}{1.50E-02} & \multicolumn{1}{c}{0.99}  \\
			\hline
		\end{tabular}}
	\end{table}

\begin{example}\label{exa:3}
	In this example, we also choose the MWG finite element space \eqref{linearspace}
	and the true solution is equal to the vector function: $\boldsymbol{u}=(x^2y^2,x(1-x)y(1-y))^T$.
\end{example}

We report the errors of $\interleave\boldsymbol{u}-\boldsymbol{u}_h\interleave$ in Table \ref{Tab:4} by fixing different mesh sizes $h$ to test the numerical example. And we observe that $\interleave\boldsymbol{u}-\boldsymbol{u}_h\interleave$ is first order, which indicate that the MWG algorithm \eqref{Eq:MWG} also have first order for the non-homogeneous boundary condition problem.


%
	\begin{table}[H]
		\centering\caption{ Convergence for Example \ref{exa:3}  with non-homogeneous boundary condition}\label{Tab:4}.
		\setlength{\tabcolsep}{12mm}{
			\begin{tabular}{{p{1cm}p{2cm}p{2cm}}}\hline
			\multirow{2}{*} {$h$}& \multicolumn{2}{c}{$\interleave\boldsymbol{u}-\boldsymbol{u}_{h}\interleave$}
			\\\cline { 2 - 3 }
			& \multicolumn{1}{c}{Error}  & \multicolumn{1}{c}{order}  \\ \hline
			1/4 & \multicolumn{1}{c}{2.17E-01} &  \multicolumn{1}{c}{N/A}  \\
			1/8  & \multicolumn{1}{c}{1.30E-01}  & \multicolumn{1}{c}{0.74}
			\\
			1/16& \multicolumn{1}{c}{7.11E-02} & \multicolumn{1}{c}{0.86}  \\
			1/32& \multicolumn{1}{c}{3.73E-02} & \multicolumn{1}{c}{0.93}  \\
			1/64& \multicolumn{1}{c}{1.91E-02} & \multicolumn{1}{c}{0.96}  \\
			1/128& \multicolumn{1}{c}{9.66E-03} & \multicolumn{1}{c}{0.98}  \\
			1/256& \multicolumn{1}{c}{4.86E-03} & \multicolumn{1}{c}{0.99}  \\
			\hline
		\end{tabular}}
	\end{table}

\section*{Acknowledgements}

The authors are supported by National Natural Science Foundation of China (Grant No. 12071160). The second and third authors are supported by The Guangdong Basic and Applied Basic Research Foundation (Grant No. 2019A1515010724), Characteristic Innovation Projects of Guangdong colleges and universities (Grant No. 2018KTSCX044), and The General Project topic of Science and Technology in Guangzhou, China (Grant No. 201904010117). The third author is supported by National Natural Science Foundation of China (Grant No. 12101147).


\end{document}